\documentclass[11pt]{amsart}
\usepackage{geometry}
\usepackage{amsmath, amssymb}
 \usepackage{amsmath,amscd}
\usepackage{amsfonts}
\usepackage{mathrsfs}
\usepackage[arrow,matrix,curve,cmtip,ps]{xy}
\usepackage{setspace}\usepackage[dvipsnames]{xcolor}

\usepackage{pb-diagram} 
\usepackage{pb-xy}
\usepackage{graphicx}

\usepackage{amsthm}

\usepackage{color}
\usepackage{float}
\usepackage{xcolor}

\allowdisplaybreaks

\makeatletter
\newtheorem*{rep@theorem}{\rep@title}
\newcommand{\newreptheorem}[2]{%
\newenvironment{rep#1}[1]{%
 \def\rep@title{#2 \ref{##1}}%
 \begin{rep@theorem}}%
 {\end{rep@theorem}}}
\makeatother

\newtheorem{theorem}{Theorem}
\newreptheorem{theorem}{Theorem}
\newreptheorem{proposition}{Proposition}
\newtheorem{lemma}[theorem]{Lemma}
\newtheorem{proposition}[theorem]{Proposition}
\newtheorem{corollary}[theorem]{Corollary}

\newtheorem*{theorem*}{Theorem}
\newtheorem*{proposition*}{Proposition}
\theoremstyle{remark}

\newtheorem{definition}[theorem]{Definition}

\newtheorem{example}[theorem]{Example}
\newtheorem{question}{Question}

\newcommand{\lk}{\operatorname{lk}}

\newcommand{\M}{\mathcal{M}}
\newcommand{\N}{\mathbb{N}}
\newcommand{\Z}{\mathbb{Z}}
\newcommand{\R}{\mathbb{R}}

\newcommand{\bdry}{\partial}

\newcommand{\Sum}{\displaystyle\sum}

\newcommand{\Prod}{\displaystyle\prod}
\newcommand{\Oint}{\displaystyle\oint}
\newcommand{\Int}{\displaystyle\int}

\newcommand{\ceil}[1]{\left\lceil #1 \right\rceil}

\newcommand{\pref}[1]{(\ref{#1})}

\begin{document}

\title[Clasp number and C-Complexes]{The C-complex clasp number of links}

\author{Jonah Amundsen}
\address{Department of Mathematics, University of Wisconsin-Eau Claire, Hibbard Humanities Hall 508,  Eau Claire WI 54702-4004}
\email{amundsjj3573@uwec.edu }

\author{Eric Anderson}
\address{Department of Mathematics, University of Wisconsin-Eau Claire, Hibbard Humanities Hall 508,  Eau Claire WI 54702-4004}
\email{andersew1951@uwec.edu }

\author{Christopher William Davis}
\address{Department of Mathematics, University of Wisconsin-Eau Claire, Hibbard Humanities Hall 508,  Eau Claire WI 54702-4004}
\email{daviscw@uwec.edu}

\author{Daniel Guyer}
\address{Department of Mathematics, University of Wisconsin-Eau Claire, Hibbard Humanities Hall 508,  Eau Claire WI 54702-4004}
\email{guyerdm7106@uwec.edu}

\date{\today}

\subjclass[2010]{}


\begin{abstract} 
In the 1980's Daryl Cooper introduced the notion of a C-complex (or clasp-complex) bounded by a link and explained how to compute signatures and polynomial invariants using a C-complex.  Since then this was extended by works of Cimasoni, Florens, Mellor, Melvin, Conway, Toffoli, Friedl, and others to compute other link invariants.  Informally a C-complex is a union of surfaces which are allowed to intersect each other in clasps. The purpose of the current paper is to study the minimal number of clasps amongst all C-complexes bounded by a fixed link $L$.  This measure of complexity is related to the number of crossing changes needed to reduce $L$ to a boundary link.  We prove that if $L$ is a 2-component link with nonzero linking number, then the linking number determines the minimal number of clasps amongst all C-complexes.   In the case of 3-component links, the triple linking number provides an additional lower bound on the number of clasps in a C-complex.
\end{abstract}

\maketitle

\section{Introduction}


There is a generalization of a Seifert surface to the setting of links called a \textbf{C-complex} or \textbf{clasp-complex} originally defined by Cooper \cite{Cooper82, CooperThesis}.  Informally, if $L = L_1\cup\dots\cup L_n$ is an $n$-component link then a C-complex for $L$ is a collection of Seifert surfaces, $F=F_1\cup\dots\cup F_n$ for the components of $L$ which are allowed to intersect, but only in clasps.  See Figure~\ref{fig:clasps} for a local picture of a clasp and Figures~\ref{fig: 2-component example} and \ref{fig:examples} for some examples of C-complexes.    See also Definition~\ref{defn: C-complex}.  

\begin{figure}
\begin{picture}(180,100)
\put(0,10){\includegraphics[height=.2\textwidth]{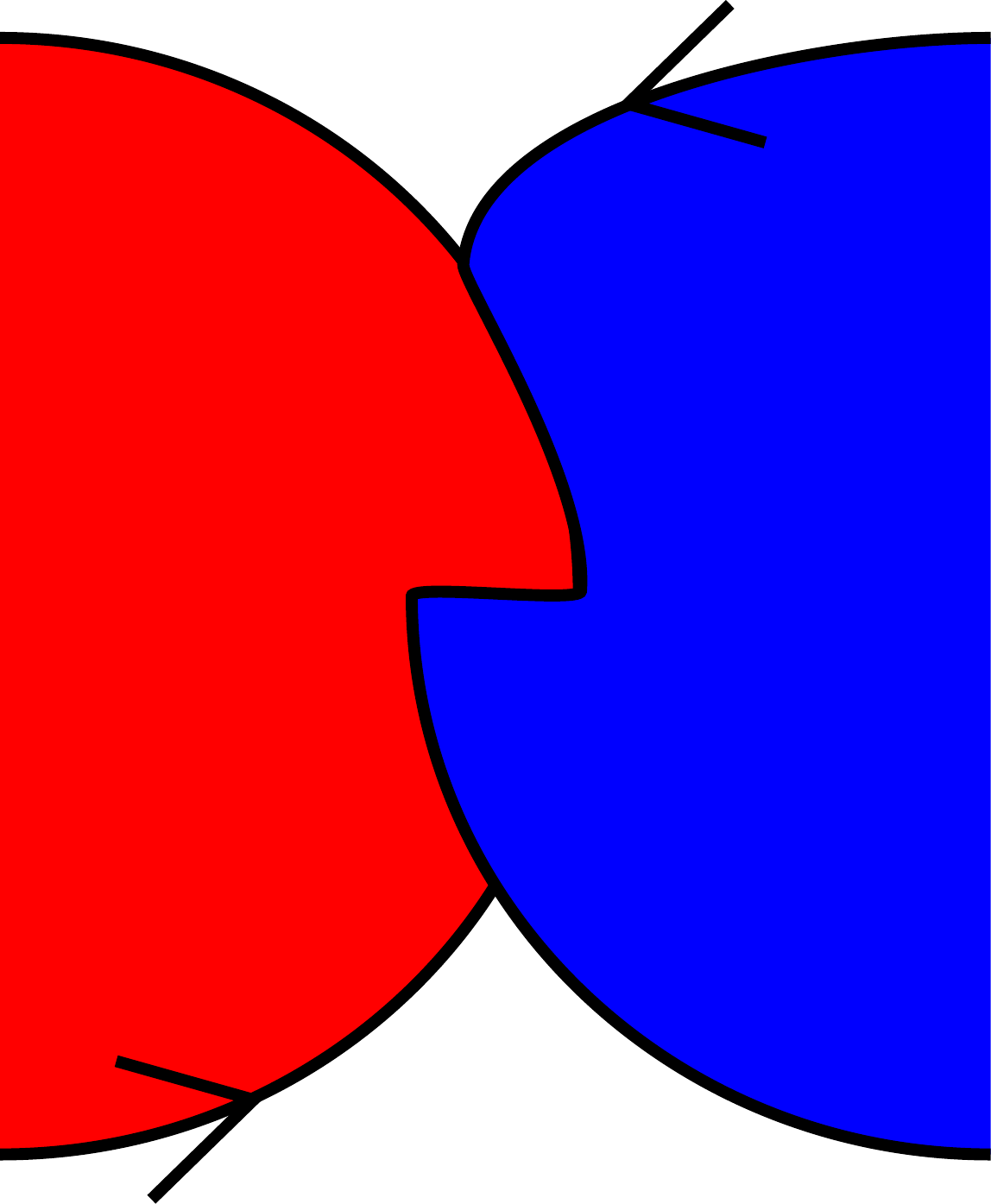}}
\put(100,10){\includegraphics[height=.2\textwidth]{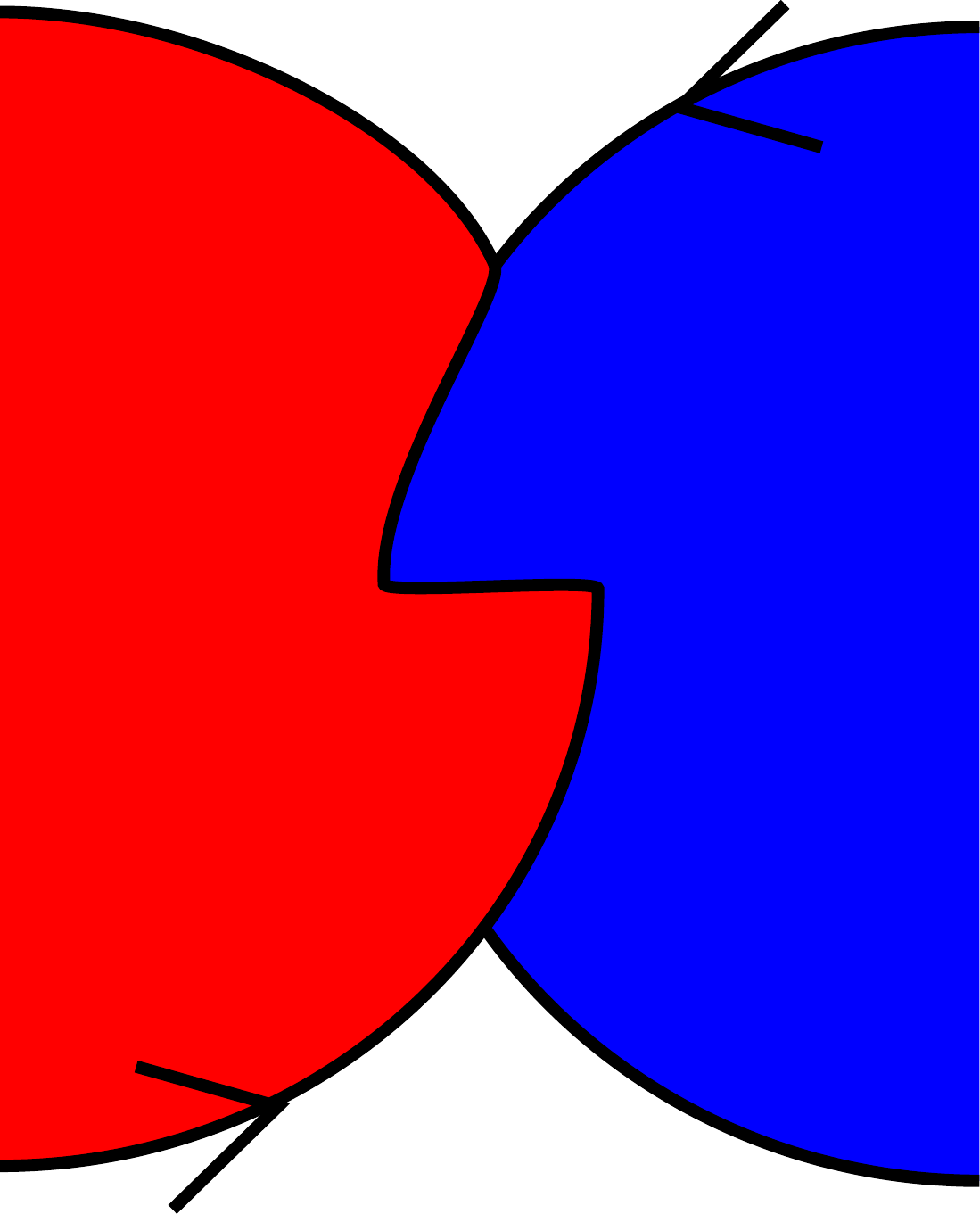}}
\end{picture}
\caption{Left: A positive clasp.  Right: A negative clasp.  
}
\label{fig:clasps}
\end{figure}


If a C-complex, $F$, for $L$ has no clasp intersections, then $F$ is a collection of disjoint Seifert surfaces for the components of $L$.   In this case $L$ is called a \textbf{boundary link} and $F$ is called a \textbf{boundary surface}.  Thus, the number of clasps in a C-complex can be used the measure how far $F$ is from being a boundary surface and so how far $L$ is from being a boundary link.  In this paper we shall study the minimal number of clasps amongst all C-complexes bounded by $L$.    This should not be confused with the clasp number introduced by Shibuya in  in \cite{Shibuya1974}.  

\begin{definition}
For a link $L$ we define the \textbf{clasp number} of $L$, $C(L)$, to be the minimum number of clasps amongst all C-complexes bounded by $L$.
\end{definition}

For a 2-component link $L=L_1\cup L_2$ and any C-complex $F=F_1\cup F_2$ bounded by $L$, the linking number, $\lk(L_1,L_2)$, can be computed as the number of positive clasps in $F$ minus the number of negative.  It follows that $C(L)\ge |\lk(L_1,L_2)|$.  Our first main result is that for most $2$-component links, $C(L) = |\lk(L_1,L_2)|$.  

\begin{theorem}\label{thm:main 2-component}
Let $L=L_1\cup L_2$ be a 2-component link.  If $\lk(L_1,L_2)\neq 0$ then $C(L) = |\lk(L_1,L_2)|$.  If $\lk(L_1,L_2)=0$ then $C(L)\in \{0,2\}$.  
\end{theorem}

\begin{figure}[h]
\begin{picture}(180,100)
\put(0,10){\includegraphics[height=.2\textwidth]{PosClasp.pdf}}
\put(100,10){\includegraphics[height=.2\textwidth]{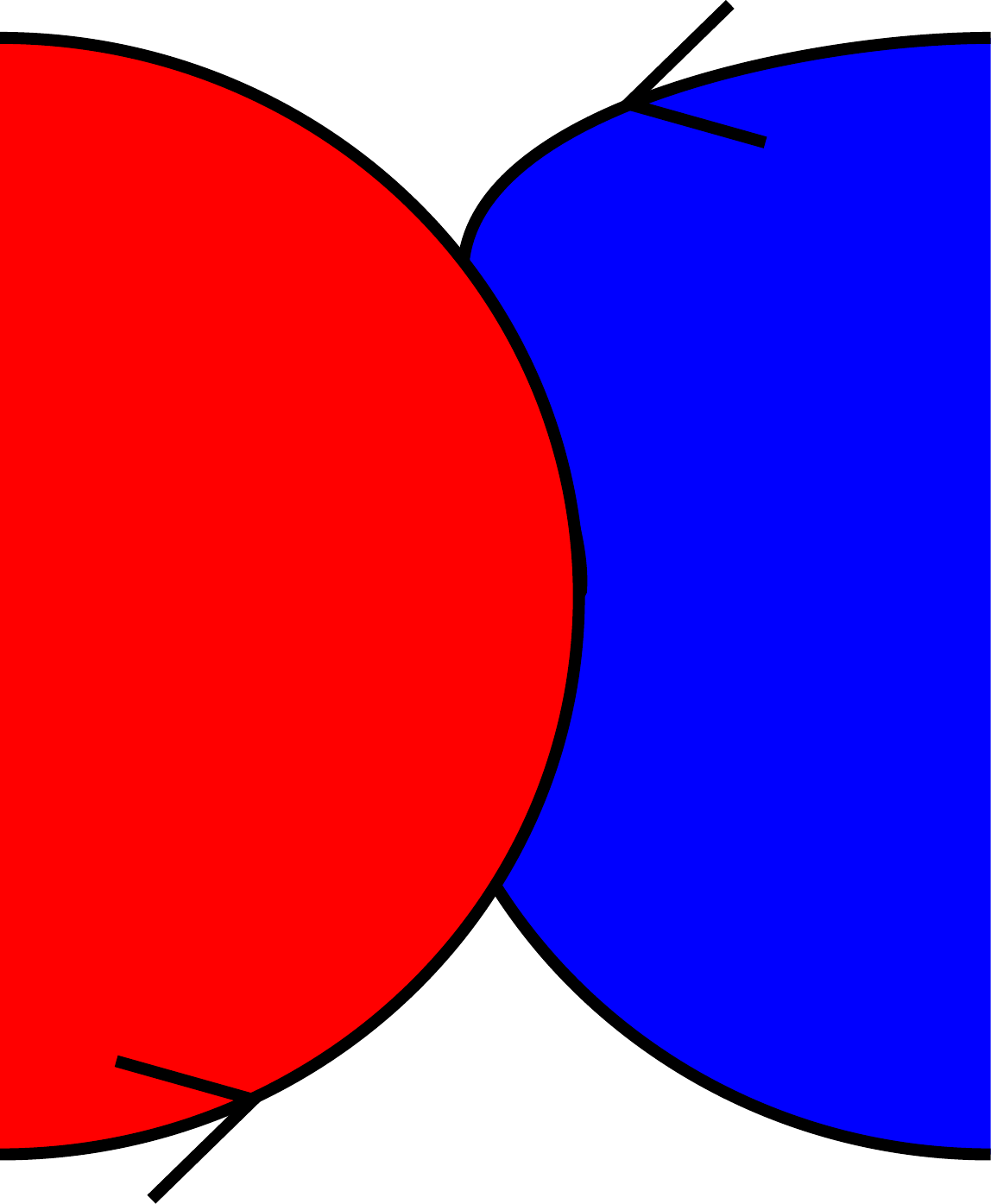}}
\end{picture}
\caption{Left: A clasp.  Right: A crossing change removing the clasp.}
\label{fig: crossing change}
\end{figure}

We mentioned that the number of clasps in a C-complex for $L$ measures how far that link is from being a boundary link.  We take a moment and make that explicit.  Any link can be reduced to a boundary link by a finite sequence of crossing changes.  Indeed, that boundary link can be taken to be the unlink.  Let $B(L)$ be the minimum number of crossing changes needed to reduce $L$ to a boundary link.  If $F$ is a C-complex for $L$ admitting $C(L)$ total clasps, then by changing a crossing at each clasp as in Figure~\ref{fig: crossing change} one reduces $F$ to a boundary surface and so $L$ to a boundary link.  Therefore
$$
B(L)\le C(L).
$$
On the other hand, changing a crossing of $L$ changes only one linking number of $L$ and that by at most $1$.  As any boundary link has vanishing pairwise linking numbers, we conclude that if $L=L_1\cup\dots\cup L_n$ is an $n$-component link then
$$
\Sum_{1\le i<j\le n }|\lk(L_i,L_j)| \le B(L).
$$

According to Theorem~\ref{thm:main 2-component} if $L=L_1\cup L_2$ has only two components and $\lk(L_1,L_2)\neq 0$ then $C(L) = |\lk(L_1,L_2)|$.  Thus, the discussion of the preceding paragraph yields the following corollary

\begin{corollary}\label{cor: 2-component}
Let $L=L_1\cup L_2$ be a 2-component link.  If $\lk(L_1,L_2)\neq 0$ then there exists a sequence of $|\lk(L_1,L_2)|$ crossing changes reducing $L$ to a boundary link.  If $\lk(L_1,L_2)= 0$ then either $L$ is a boundary link or there exists a sequence of $2$ crossing changes reducing $L$ to a boundary link. 
\end{corollary}

\begin{example}
In order to demonstrate that Theorem~\ref{thm:main 2-component} and Corollary~\ref{cor: 2-component} are surprising, consider the link of Figure~\ref{fig: 2-component example}.  The depicted C-complex has three clasps.  Since $\lk(L_1,L_2)=1$, there exists a C-complex bounded by $L$ with a single clasp and perhaps more surprisingly there exists a single crossing change reducing $L$ to a boundary link.  
\end{example}

\begin{figure}[h]
\begin{picture}(190,70)
\put(0,0){\includegraphics[height=.1\textheight]{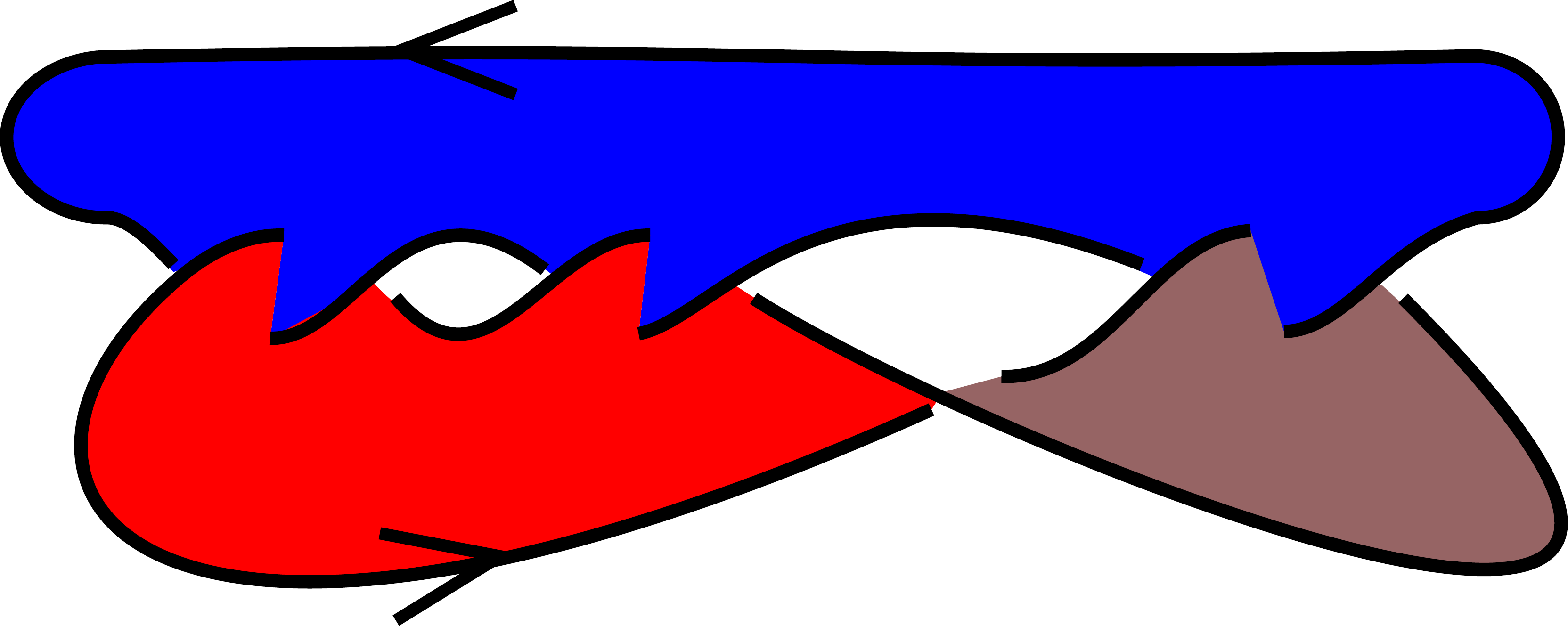}}
\end{picture}
\caption{A 2-component link with linking number $1$.}\label{fig: 2-component example}
\end{figure}

According to Theorem~\ref{thm:main 2-component}, the linking number  determines the clasp number of 2-component links.  This behavior does not extend to links of more than 2 components.  In \cite{Milnor1950}, Milnor introduced a family of higher order linking invariants.  The first of these, the triple linking number, $\mu_{ijk}$, is well defined when the pairwise linking numbers vanish and  measures the linking of the $i$'th, $j$'th, and $k$'th components.  According to Mellor-Melvin \cite{MellorMelvin2003}, $\mu_{ijk}(L)$ can be computed in terms of the clasps of a C-complex bounded by $L$.  Thus, it comes as no surprise that $\mu_{123}(L)$ can be used to deduce a bound on $C(L)$. We explicitly compute this bound. 

\begin{theorem}\label{thm:main 3-component}
Let $L=L_1\cup L_2\cup L_3$ be a 3-component link with vanishing pairwise linking numbers.  Then $C(L)\ge 2\ceil{2\sqrt{|\mu_{123}(L)|/3}}$.  Here $\ceil{-}$ is the ceiling function.  
\end{theorem}

In order to illustrate the power of this theorem we compute the clasp number of some examples.  The Boromean rings, denoted $BR$, has $\mu_{123}(BR)=1$ and so by Theorem~\ref{thm:main 3-component}, $C(BR)\ge 4$.  Figure~\ref{fig:examples}~(a) depicts a C-complex bounded by $BR$ with four clasps.  Thus, $C(BR)=4$.
For any $n\in \N$ the generalized Boromean rings $BR^n$ of Figure~\ref{fig:examples}~(b) bounds a C-complex with $4n$ clasps and has $\mu_{123}(BR^n) = n^2$.  We make this computation in Proposition~\ref{prop:compute}.  As a consequence we get the following corollary, producing links with vanishing pairwise linking numbers and arbitrarily large clasp number.  

\begin{figure}
\begin{picture}(320,240)
\put(0,50){\includegraphics[height=.18\textheight]{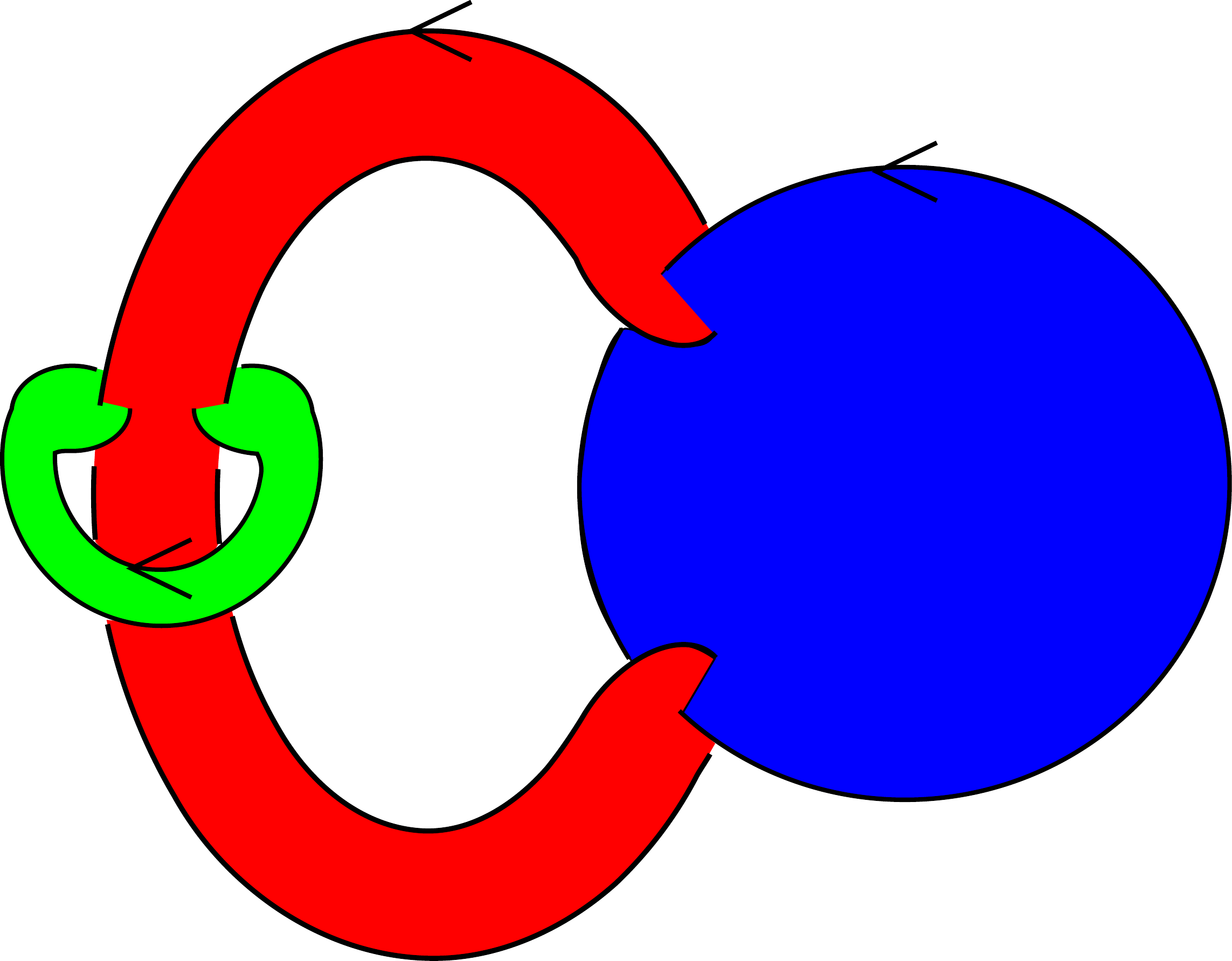}}
\put(40,40){{$BR_1$}}
\put(-15,90){{$BR_3$}}
\put(100,60){{$BR_2$}}
\put(55,25){(a)}

\put(160,40){\includegraphics[width=.3\textwidth]{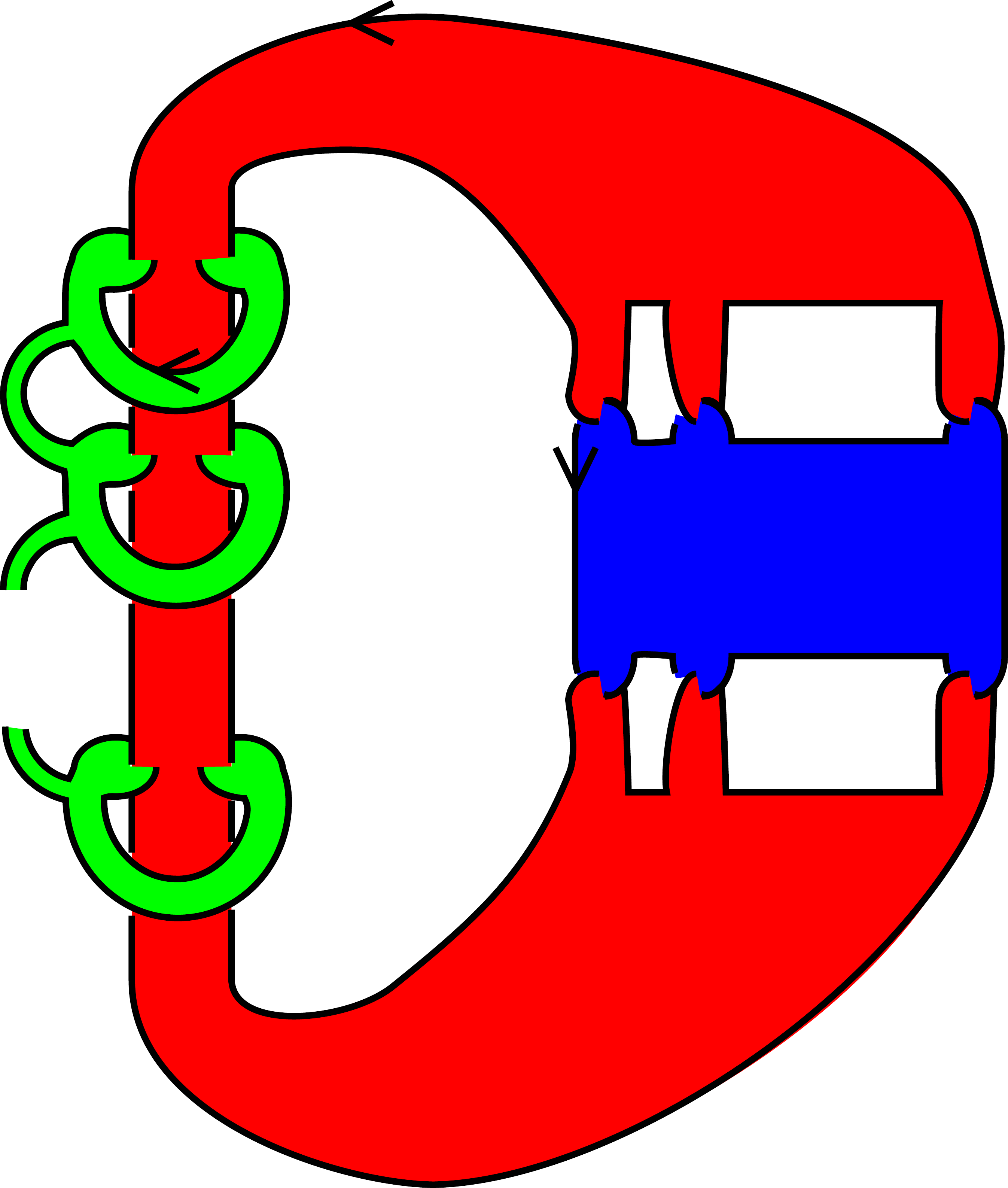}}
\put(270,152){$\dots$}
\put(270,105){$\dots$}
\put(160,110){$\vdots$}

\put(295,180){$BR^n_1$}

\put(148,80){$BR^n_3$}

\put(302,125){$BR^n_2$}

\put(215,25){(b)}

%
%
%
%
%
\end{picture}
\caption{(a) A C-complex bounded by the Boromean rings.  (b)  A C-complex bounded by the generalized Boromean rings, $BR^n$.  
}\label{fig:examples}
\end{figure}
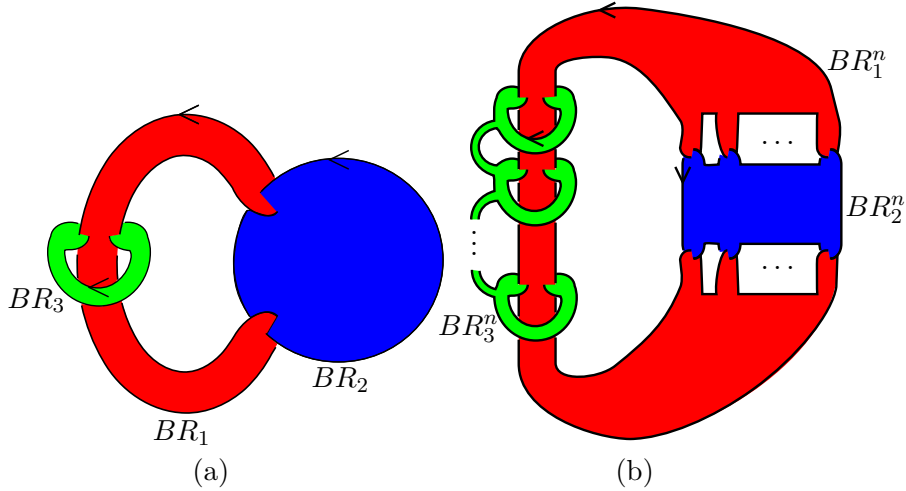

%
%
%
%

\begin{corollary}\label{cor: 3-component}
For any $n\in \N$, consider generalized Boromean rings $BR^n$ of Figure~\ref{fig:examples}~(b).  The pairwise linking numbers of $BR^n$ vanish and yet $2\ceil{2n/\sqrt{3}}\le C(L) \le 4n$.  
\end{corollary}

  In \cite{MellorMelvin2003}, Mellor-Melvin provides a means of computing $\mu_{123}(L)$ in terms of any collection of Seifert surfaces for the components of $L$.  We shall use this result in the special case of a C-complex.  While a more complete description appears in Section~\ref{sect: triple linking}, we recall it informally now.  Start with a link $L=L_1\cup L_2\cup L_3$ bounding a C-complex $F=F_1\cup F_2\cup F_3$, follow a component $L_k$ of $L$, and record a word $w_k(F)$ in $x_1^{\pm1},x_2^{\pm1}, x_3^{\pm1}$ capturing the order and sign of the clasps $L_k$ encounters.  Set $e_{ij}(w_k(F))\in \Z$ to be the signed count of the number of $x_i$'s appearing in $w_k$ before an $x_j$. The triple linking number is given by $\mu_{123}(L) = e_{12}(w_3(F))+e_{23}(w_1(F))+e_{31}(w_2(F))$. 
  
  A technical result we use in our proof of Theorem~\ref{thm:main 3-component} is a new geometric strategy to compute $e_{ij}(w)$.  For any word $w$ in letters $x_1^{\pm1}, x_2^{\pm1}, x_3^{\pm1}$ and any $i,j\in \{1,2,3\}$ construct a piecewise linear curve $\gamma_{ij}(w)$ in $\R^2$ as follows.  Start at the origin $(0,0)$.  Each time you see an $x_i$ (respectively $x_i^{-1}$, $x_j$, $x_j^{-1}$) in $w$ travel right (respectively left, up, down) a length of 1.  The following reveals that $e_{ij}(L)$ is the area enclosed by this curve.  
  
  \begin{theorem}\label{thm:eij in terms of area}
  Let $w = \Prod_{n=1}^m x_{i_n}^{\epsilon_n}$ be a word in letters $x_1^{\pm1}, x_2^{\pm1}, x_3^{\pm1}$.  For any $i\neq j\in \{1,2,3\}$,  $e_{ij}(w) = \Oint_{\gamma_{ij}(w)} x\,dy.$
Additionally, if $\gamma_{ij}(w)$ is a simple closed curve with counterclockwise orientation, then $e_{ij}(w)$ is the area enclosed by $\gamma_{ij}(w)$.
  \end{theorem}

\subsection{Questions}

Theorem~\ref{thm:main 2-component} states that any 2-component link with nonzero linking number has a C-complex admitting precisely $|\lk(L_1,L_2)|$ clasps.  However, our proof makes no attempt to minimize the first Betti number of the C-complex, which is the measure of complexity most directly accessible using the tools like Alexander polynomial or signature \cite{Cimasoni2004, CimFlo2008}.  We pose the following question.

\begin{question}
Suppose that $L = L_1\cup L_2$ is a 2-component link with nonzero linking number.  Amongst all C-complexes $F$ bounded by $L$ admitting precisely $|\lk(L_1, L_2)|$ clasps, what is the minimal value for $\beta_1(F)$?  Is it possible to simultaneously minimize the number of clasps in $F$ as well as $\beta_1(F)$?
\end{question}

Theorem~\ref{thm:main 2-component} almost completely determines $C(L)$ for 2-component links.  Theorem~\ref{thm:main 3-component} concludes that $C(L)\ge 2\ceil{2\sqrt{|\mu_{123}(L)|/3}}$ for three component links with vanishing linking numbers.  One might ask if equality holds.

\begin{question}
Let $L=L_1\cup L_2\cup L_3$ be a 3-component link with vanishing pairwise linking numbers and $\mu_{123}(L)\neq 0$.  Does it follows that $C(L) =  
2\ceil{2\sqrt{|\mu_{123}(L)|/3}}$?
\end{question}

More specifically, for any $n\in \N$, consider the generalized Boromean rings $BR^n$ of Figure~\ref{fig:examples}~(b).  Corollary~\ref{cor: 3-component} concludes that $2\ceil{2n/\sqrt{3}}\le C(BR^n)\le 4n$.  When $n=2$ this gives $6\le C(BR^2)\le 8$. 

\begin{question} What is $C(BR^n)$?
\end{question}

%
%
%
%
%

One might ask about the clasp number of links of more than three components.

\begin{question}
Let $n\ge3$ and let $L=L_1\cup \dots \cup L_n$ be an $n$-component link with vanishing pairwise linking numbers and $\mu_{ijk}(L)\neq 0$ for some $i,j,k$.  Is there a formula for $C(L)$ in terms of the set of all triple linking numbers of $L$?
\end{question}

In the case of links of more than 2 components with nonvanishing pairwise linking numbers, the triple linking numbers are not well defined. Instead by \cite{DNOP} there is a \textbf{total triple linking number} recording all of the individual triple linking numbers taking values in some quotient $\M$.  

\begin{question}
Let $L=L_1\cup \dots \cup L_n$ be an $n$-component link with either a nonvanishing pairwise linking number or a nonvanishing triple linking number.  Is there a formula for $C(L)$ in terms of the linking numbers and the total triple linking number?
\end{question}


\section{C-complexes and the proof of Theorem~\ref{thm:main 2-component}}

Throughout this paper all knots will be smoothly embedded curves in $S^3$, and all surfaces will be smoothly embedded in $S^3$, compact, connected, and oriented.  A smoothly embedded compact oriented surface with boundary equal to a knot $K$ will be called a \textbf{Seifert surface} for $K$.  We begin by recalling the formal definition of a C-complex.    

\begin{definition}\label{defn: C-complex}\cite[Section 2.1]{CimFlo2008}
Given a link, $L = L_1\cup\dots\cup L_n$, a C-complex for $L$ is a collection of Seifert surfaces, $F = F_1\cup\dots\cup F_n$, for the components of $L$ which may intersect transversely with the following constraints:
\begin{enumerate}
\item For each $i,j\in\{1,\dots, n\}$, $F_i\cap F_j$ is a union of simple arcs running from a point in $L_i=\bdry F_i$ to a point in $L_j = \bdry F_j$.  These arcs are called \emph{clasps}.  See Figure~\ref{fig:clasps}. 
\item For any three distinct $i,j,k$, $F_i\cap F_j\cap F_k=\emptyset$.  
\end{enumerate}
\end{definition}

For this section we restrict to the case that the number of components is $n=2$.  A clasp between $F_1$ and $F_2$ has endpoints given by a point in $F_1\cap L_2$ and a point in $L_1\cap F_2$.  We call a clasp \textbf{positive} (or \textbf{negative}, respectively) if these points of intersection are positive (or negative, respectively).  See Figure~\ref{fig:clasps} for a local picture.   If $F_2$ is any Seifert surface for $L_2$, then the linking number $\lk(L_1,L_2)$ is given by counting with sign how many times $L_1$ passes through $F_2$.  See for example \cite[5D]{Rolfsen}.  If $F_1\cup F_2$ is a C-complex for $L_1\cup L_2$ then this is precisely the same as the signed count of the clasps shared by $F_1$ and $F_2$.  We are now ready to prove Theorem~\ref{thm:main 2-component}.


\begin{reptheorem}{thm:main 2-component}
Let $L=L_1\cup L_2$ be a 2-component link.  If $\lk(L_1,L_2)=0$ then $C(L)\in \{0,2\}$.  If $\lk(L_1,L_2)\neq 0$ then $C(L) = |\lk(L_1,L_2)|$
\end{reptheorem}

\begin{proof}[Proof of Theorem \ref{thm:main 2-component}]
Let $L=L_1\cup L_2$ be a 2-component link and $F=F_1\cup F_2$ be any C-complex bounded by $L$.  Let $c_+$ be the number of positive clasps in $F$ and $c_-$ be the number of negative clasps.  By the triangle inequality, 
$$
|\lk(L_1,L_2)| = |c_+-c_-|\le c_+ + c_-.
$$  So that $F$ has at least $|\lk(L_1,L_2)|$ many clasps.  As $F$ is an arbitrary C-complex bounded by $L$, $C(L)\ge |\lk(L_1,L_2)|$.  Thus, we need only to show that $C(L)\le |\lk(L_1,L_2)|$.  Since $C(L)$ is the minimum number of clasps amongst all C-complexes bounded by $L$, it suffices to exhibit a C-complex with precisely $|\lk(L_1,L_2)|$ clasps or 2 clasps in the case that $\lk(L_1,L_2)=0$.  Without loss of generality we shall assume that $\lk(L_1,L_2)\ge0$.

We begin by producing a pair of Seifert surfaces $F_1$ and $F_2$ for $L_1$ and $L_2$ which will have no negative clasps in their intersection but which may have some non-clasp intersections.  Let $F_1$ be any Seifert surface for $L_1$.  Suppose $F_1$ is transverse to $L_2$ and $F_1\cap L_2$ contains $n_+$ positive points of intersection and $n_-$ points of negative intersection.  If both of $n_+$ and $n_-$ are nonzero then as you follow $L_2$ you will at some point encounter a positive point of intersection with $F_1$ followed by a negative, as in Figure~\ref{fig: Tubing1}~(a).  By adding a tube to $F_1$ as in Figure~\ref{fig: Tubing1}~(b) we see a new Seifert surface bounded by $L_1$ which intersects $L_2$ in two fewer points.   Iterating, we see a Seifert surface for $L_1$, which we persist in calling $F_1$, bounded by $L_1$ which either intersects $L_2$ in only positive points or only negative points of intersection.  Thus, $n_+=0$ or $n_-=0$.  Since $n_+-n_-=\lk(L_1,L_2)\ge 0$ by assumption we see that $n_-=0$.  By the same process, we find a Seifert surface $F_2$ which intersects $L_1$ in only positive points of intersection.  

\begin{figure}[h]
\begin{picture}(200,100)
\put(0,10){\includegraphics[height=.2\textwidth]{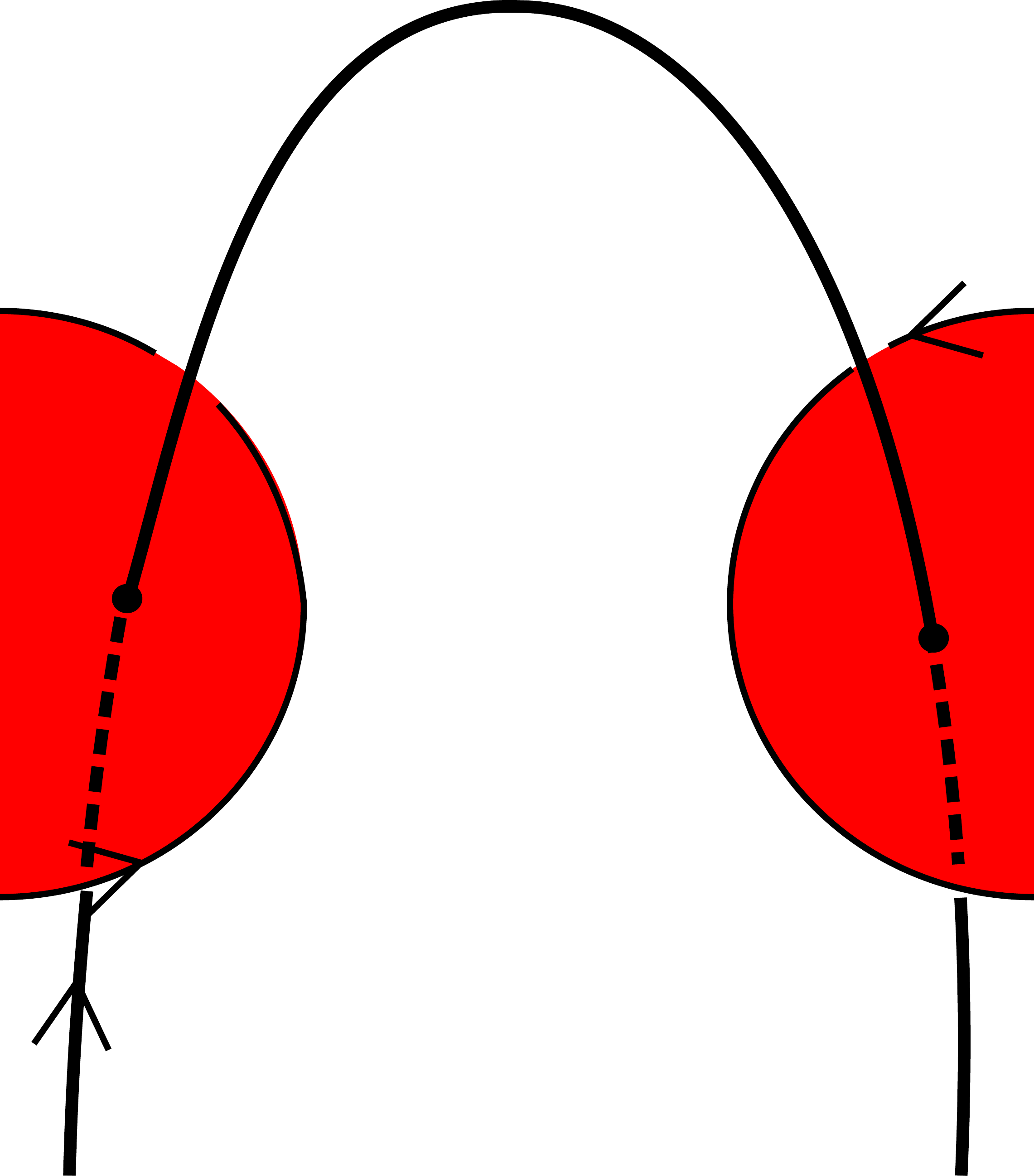}}
\put(30,0){(a)}
\put(120,10){\includegraphics[height=.2\textwidth]{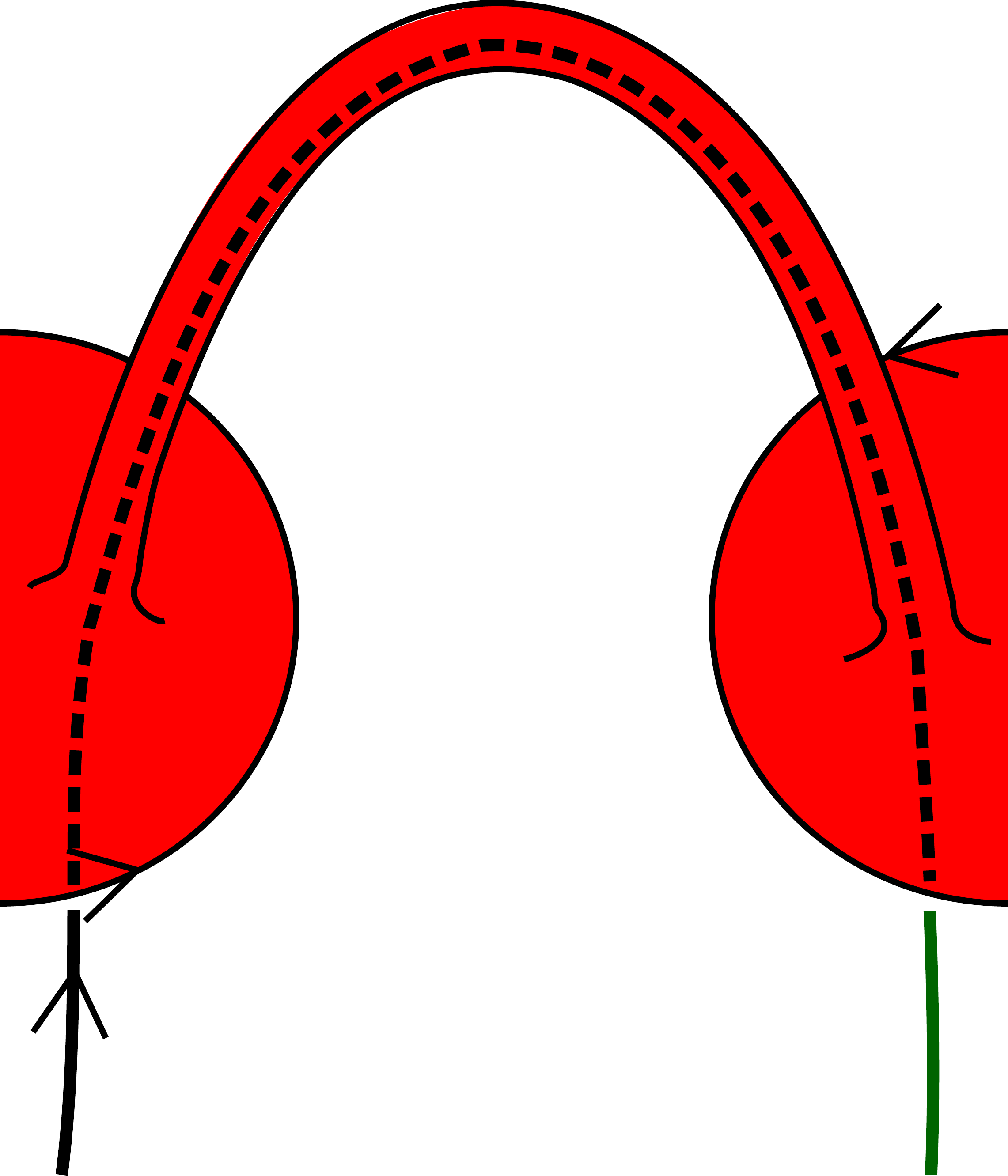}}
\put(150,0){(b)}
\end{picture}
\caption{(a) A knot $L_2$ intersecting an oriented surface $F_1$ in a positive point of intersection followed by a negative point of intersection.  (b) Adding a tube to $F_1$ removes both intersection points.}
\label{fig: Tubing1}
\end{figure}

There is no reason to expect that $F_1\cup F_2$ is a C-complex.  After a small isotopy of $F_1$ and $F_2$ we may assume that they intersect transversely.  Therefore $F_1\cap F_2$ consists of a collection of: 
\begin{itemize}
\item Arcs with one endpoint in $L_1 = \bdry F_1$ and the other in $L_2 = \bdry F_2$. (Clasps.)
\item Arcs with both endpoints in $L_1 = \bdry F_1$ or both endpoints in $L_2 = \bdry F_2$. (Ribbons.)
\item Simple closed curves interior to $F_1$ and interior to $F_2$. (Loops.)
\end{itemize}
See Figure~\ref{fig:ClaspRibbonLoop}.  Since $F_1$ has no negative points of intersection with $L_2$, there can be no negative clasps in $F_1\cap F_2$.  The endpoints of a ribbon intersection are intersection points between $F_1$ and $L_2$ (or $F_2$ and $L_1$) with opposite signs.  Since we have already arranged that there are no negative points of intersection, there can be no ribbon intersections in $F_1\cap F_2$.  Thus, $F_1\cap F_2$ consists only of loops and positive clasps.  It remains to further modify $F_1$ and $F_2$ to eliminate all loops.

\begin{figure}[h]
\begin{picture}(280,100)
\put(0,10){\includegraphics[height=.2\textwidth]{PosClasp.pdf}}
\put(32,0){(a)}
\put(120,10){\includegraphics[height=.2\textwidth]{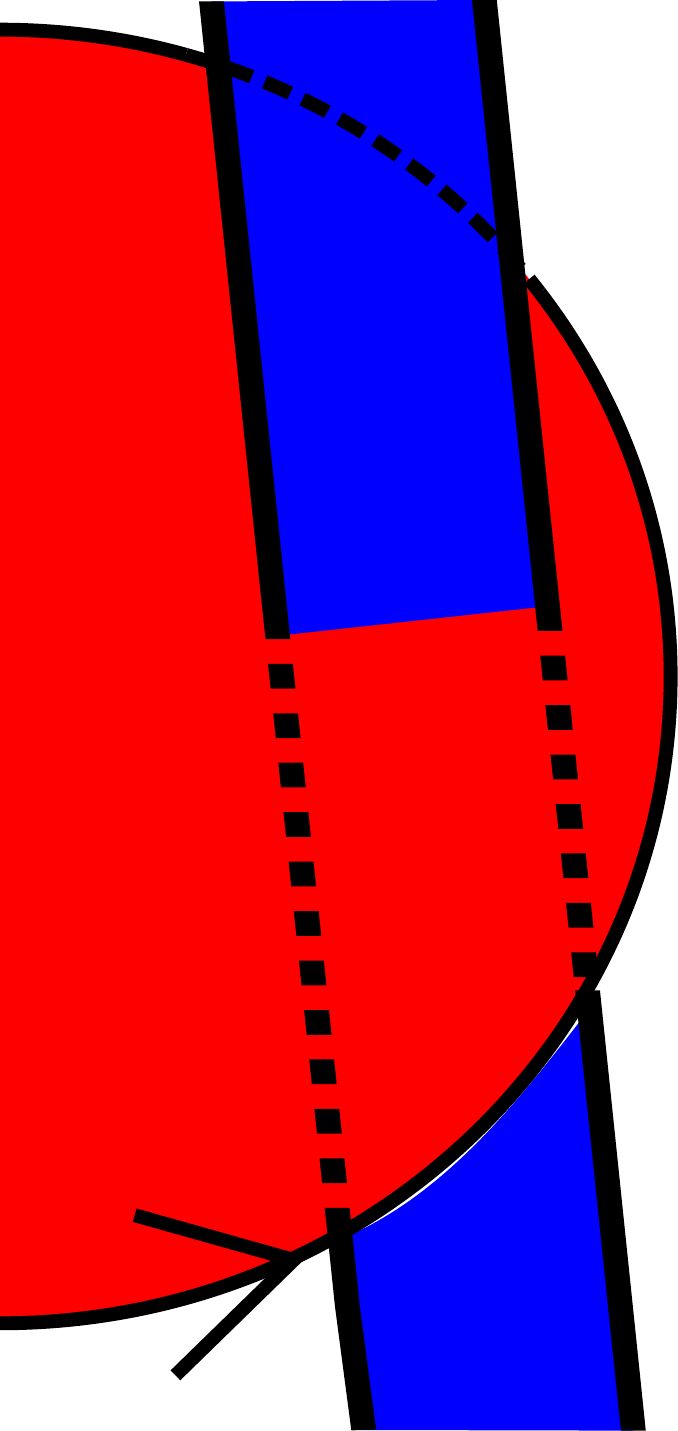}}
\put(135,0){(b)}
\put(200,10){\includegraphics[height=.2\textwidth]{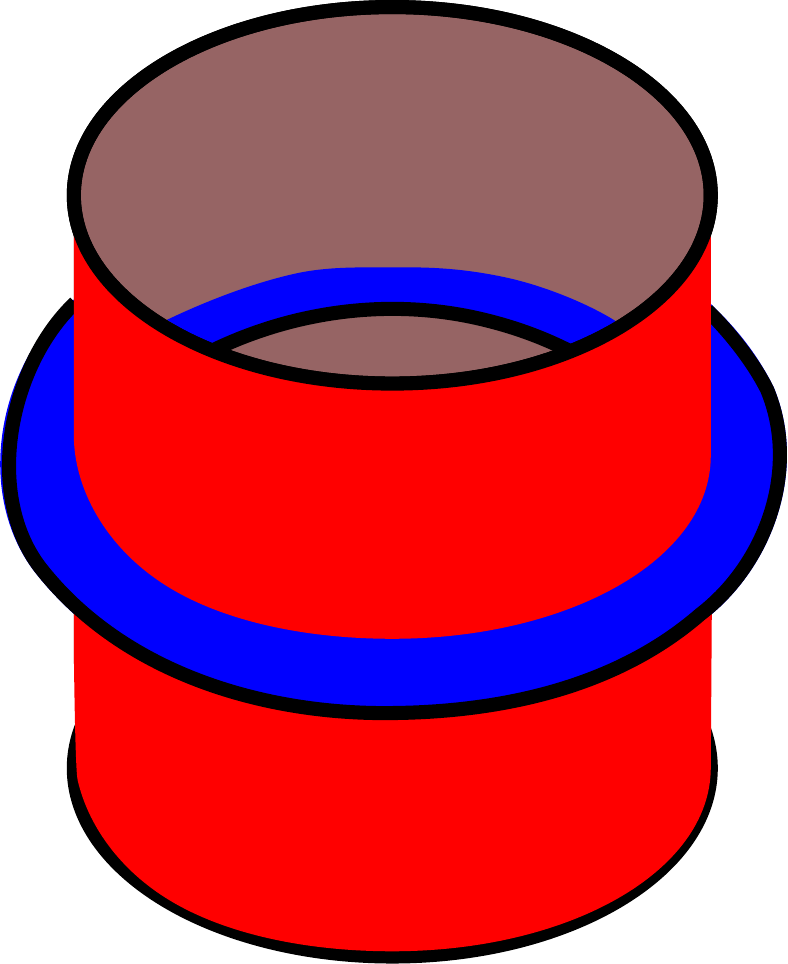}}
\put(233,0){(c)}
\end{picture}
\caption{(a) A positive clasp intersection.  (b) A ribbon intersection.  (c)  A loop intersection.}
\label{fig:ClaspRibbonLoop}
\end{figure}

Assume that $\lk(L_1,L_2)\neq0$ so that there there is at least one clasp in $F_1\cap F_2$.  Let $c$ be one such clasp.  Consider any loop intersection $\ell\subseteq F_1\cap F_2$.  Without loss of generality we may assume that there exists an arc $\alpha$ in $F_2$ running from a point in $c$ to a point in $\ell$. Moreover, we may assume that $\alpha$ connects two points pushed off from $F_1$ in the same normal direction.    Figure~\ref{fig:ClaspEatsLoop} reveals how one may add a tube to $F_1$ following $\alpha$ to combine $c$ and $\ell$ into a single simple arc.  The resulting arc has one endpoint in $L_1$ and the other in $L_2$.  In other words, it is a clasp.  Thus, we have reduced the number of loop intersections by 1 and preserved the number of clasp intersections.  Iterating, we eliminate all loop intersections and produce a C-complex for $L=L_1\cup L_2$ with number of clasps equal to $\lk(L_1,L_2)$, as claimed.

\begin{figure}
\begin{picture}(400,100)
\put(0,0){\includegraphics[height=.2\textwidth]{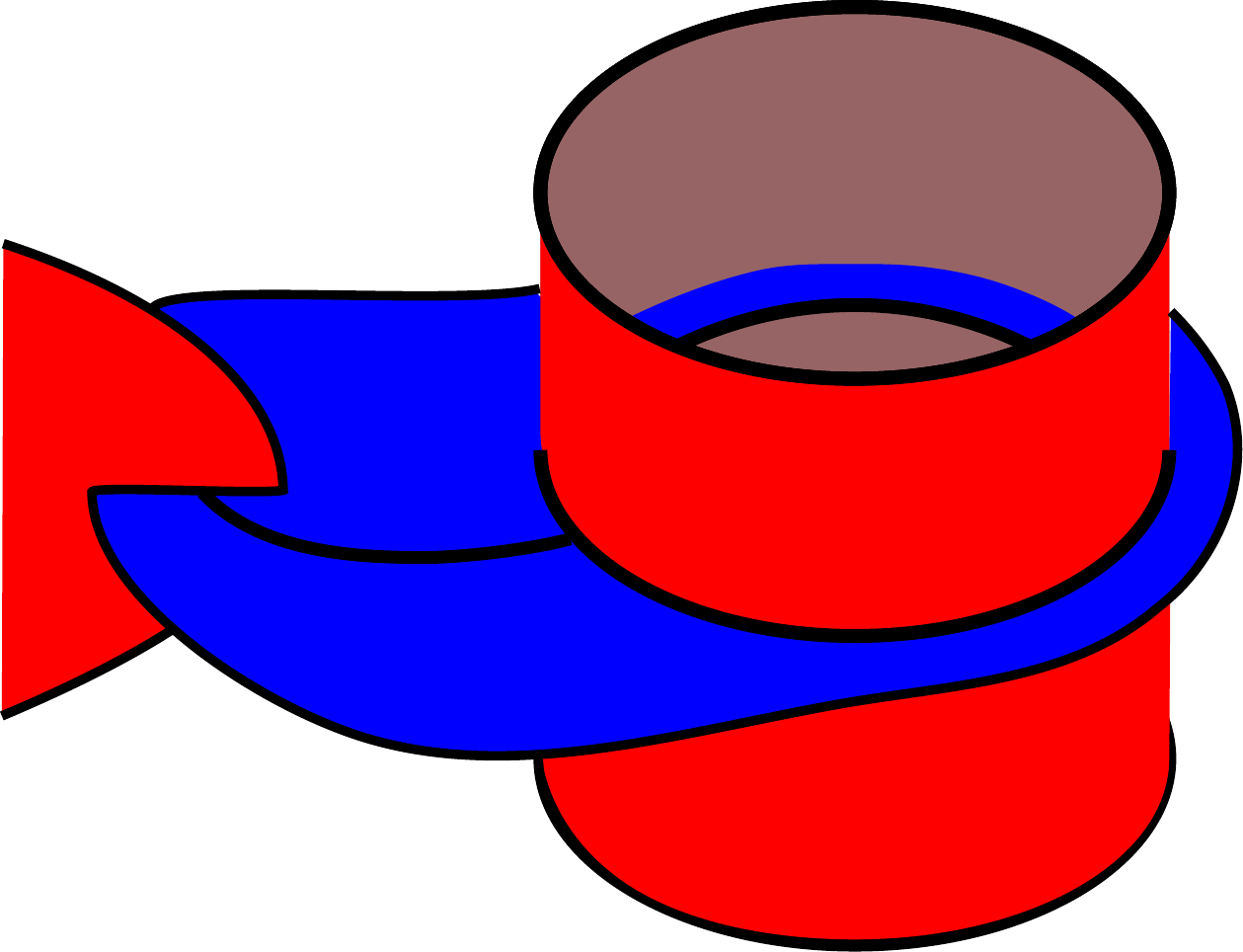}}
\put(140,0){\includegraphics[height=.2\textwidth]{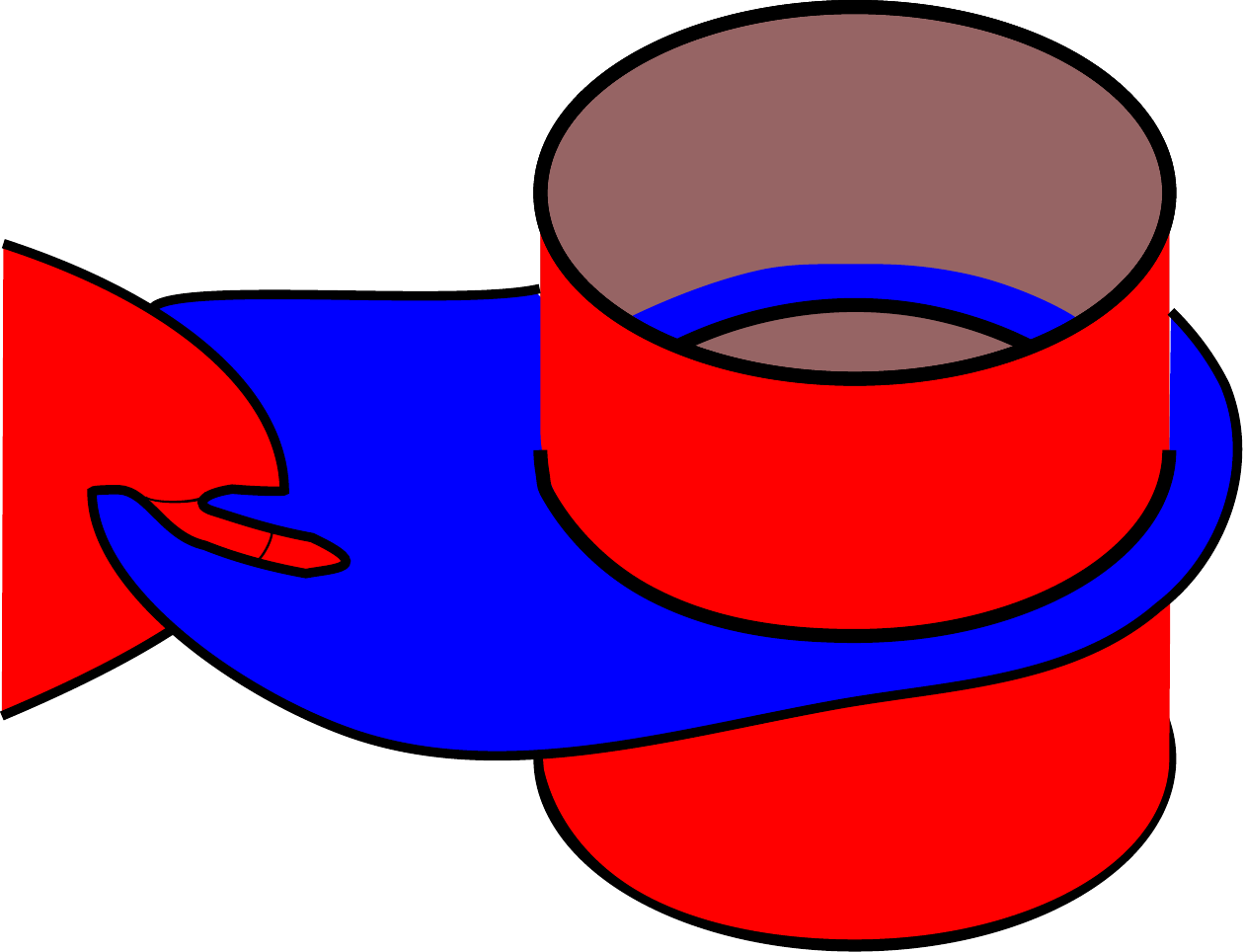}}
\put(280,0){\includegraphics[height=.2\textwidth]{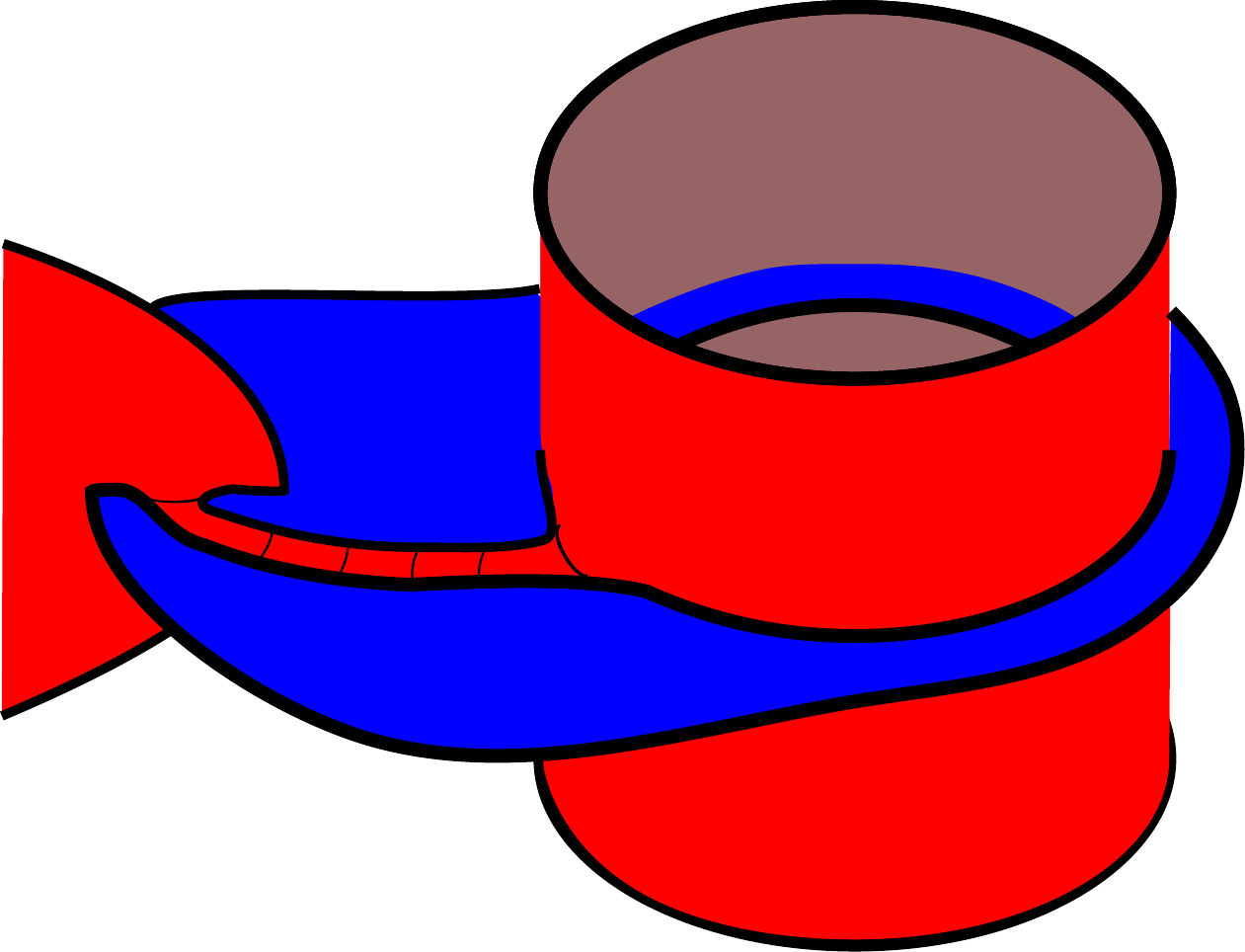}}
\end{picture}
\caption{Left: A pair of surfaces sharing a clasp and a loop intersection together with an arc running from the clasp to the loop.  Center: Perform a finger move to push the clasp intersection closer to the loop.  Right:  Tubing the clasp into the loop results in a single clasp intersection.}
\label{fig:ClaspEatsLoop}
\end{figure}

In the case that the linking number is zero, $F_1\cap F_2$ contains no clasps.  If $F_1\cap F_2$ also has no loops, then $F_1\cup F_2$ is a C-complex with no clasps and $C(L)=0$.  Otherwise, modify $F_1\cup F_2$ as in Figure~\ref{fig:CancellingClasps} to add a positive and a negative clasp.  Now we use the move of Figure~\ref{fig:ClaspEatsLoop} just as in the previous paragraph to remove all loop intersections and produce a C-complex with precisely 2 clasps, so $0\le C(L)\le 2$.  In order to see that $C(L)$ cannot be $1$, notice that since $c_+-c_-=\lk(L_1,L_2)=0$, it must be that $c_+=c_-$.  In particular, $F$ has an even number of clasps.  This completes the proof.   

\begin{figure}[h]
\begin{picture}(270,100)
\put(0,0){\includegraphics[height=.2\textwidth]{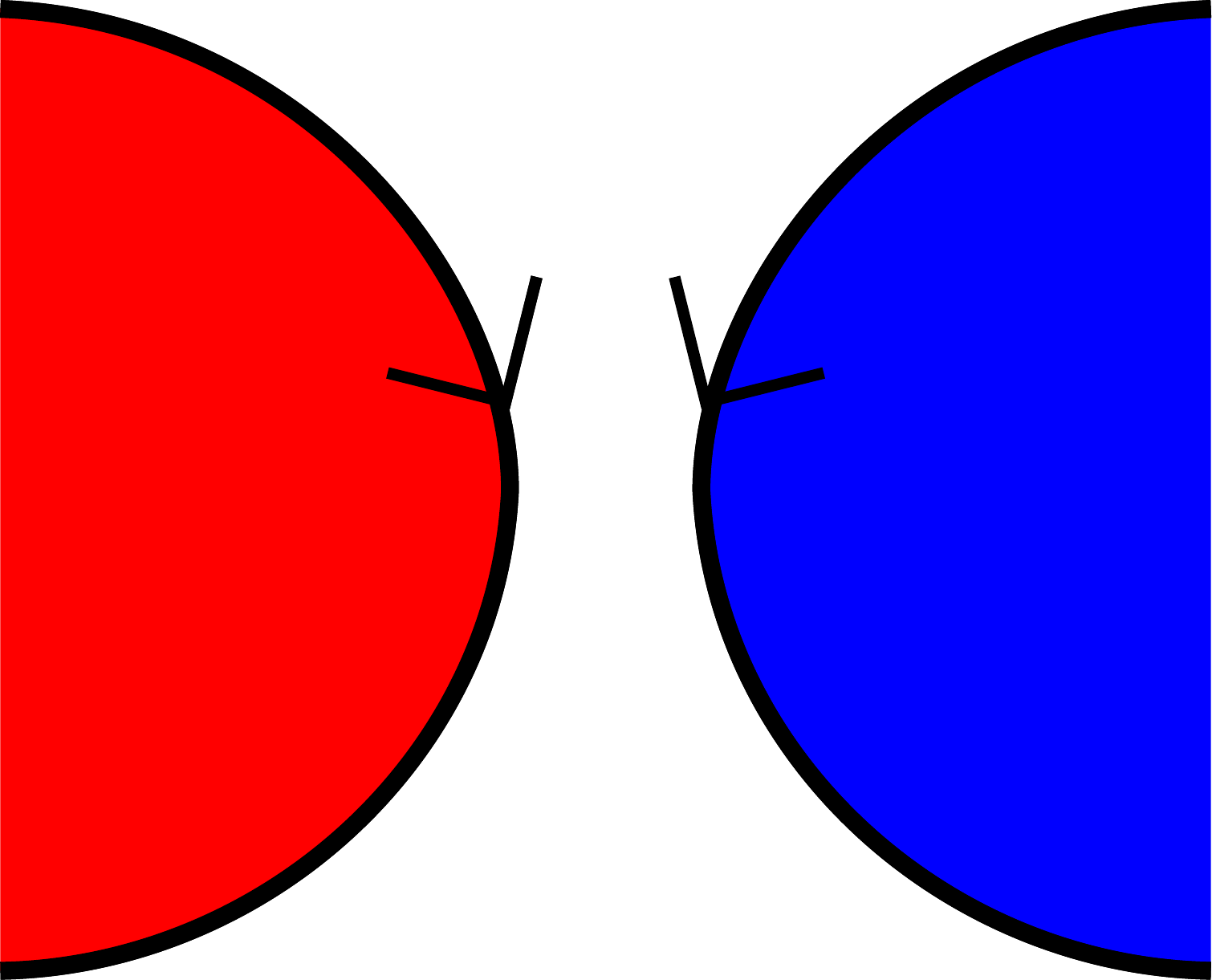}}
\put(150,0){\includegraphics[height=.2\textwidth]{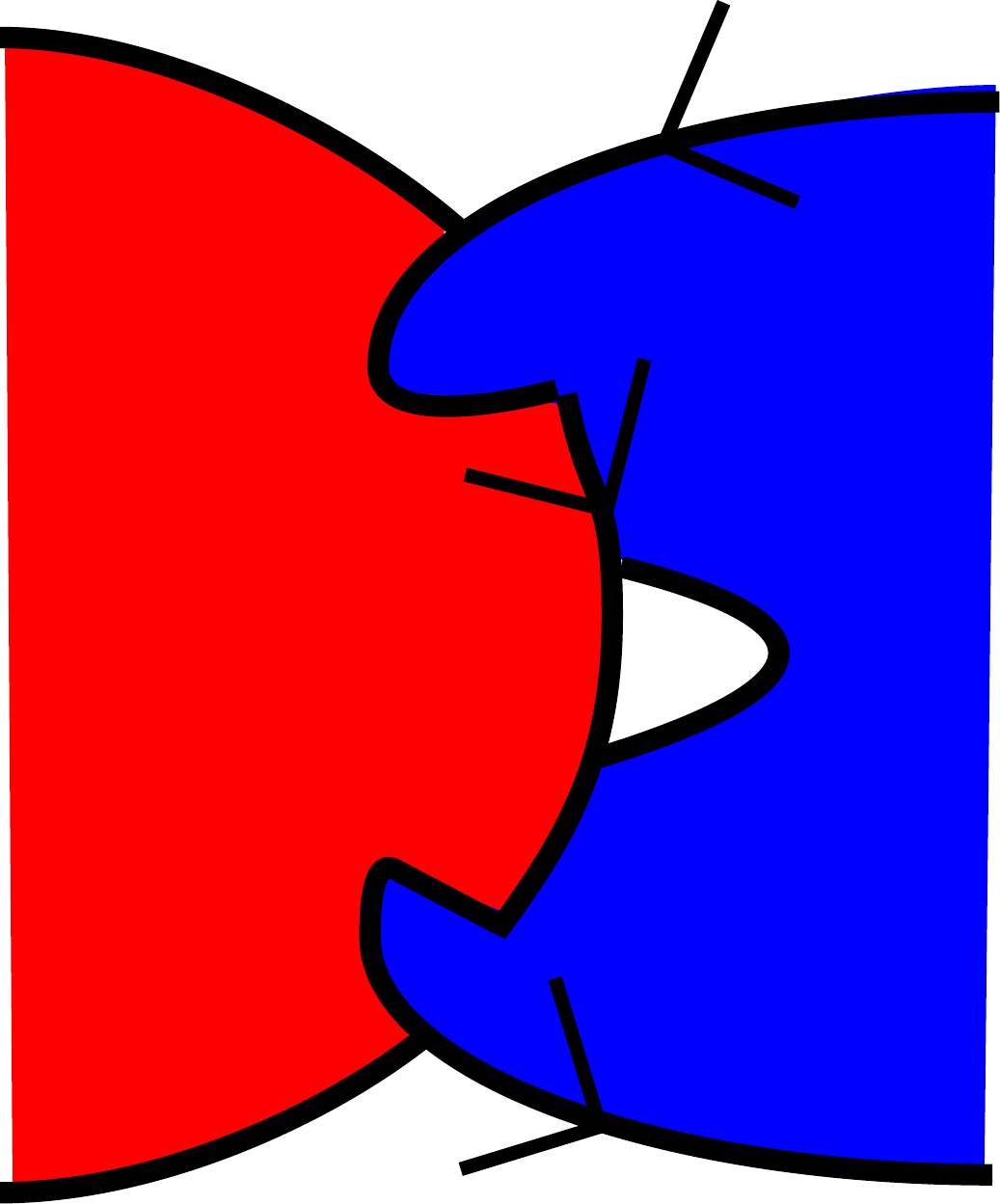}}
\end{picture}
\caption{Modifying a C-complex by inserting two cancelling clasps.}
\label{fig:CancellingClasps}
\end{figure}

\end{proof}

\section{Triple linking numbers via clasps and polyominos}\label{sect: triple linking}

In this section we recall an invariant of links called the triple linking number and provide a formula in terms of the area of a polyomino.  A \textbf{polyomino} is a region of $\R^2$ consisting of a union of closed unit squares with vertices at points in $\Z^2$.

In \cite{MellorMelvin2003} Mellor-Melvin produces a formula for the triple linking number for any union of Seifert surfaces for the components of $L$.  We shall recall it in the special case of a C-complex.  Let $L = L_1\cup\dots\cup L_n$ be an $n$-component link and $F=F_1\cup\dots
\cup F_n$ be a C-complex bounded by $L$.  We associate to each $k=1,\dots, n$ a word $w_k(F)$ called a \textbf{clasp word} as follows.  Pick a basepoint $p_k$ on $L_k$ and follow $L_k$ in the positive direction starting at $p_k$.  Record an $x_j$ whenever $L_k$ crosses through $F_j$ at a positive clasp and $x_j^{-1}$ when $L_k$ crosses $F_j$ at a negative clasp.  Let $e_{ij}(w_k(F))$ be given by counting with sign how often in $w_k(F)$ $x_i$ appears before $x_j$.  More formally, if $w_k(F)=\Prod_{v=1}^m x_{i_v}^{\epsilon_v}$ then 
\begin{equation}
\label{eqn:eij}e_{ij}(w_k(F)) = \Sum_{v=1}^m\Sum_{u=1}^v \delta(i_{u},i)\delta(i_{v},j)\epsilon_{u}\epsilon_{v}.
\end{equation}
Here $\delta(a,b)=\begin{cases} 1&\text{if }a=b\\0&\text{otherwise}\end{cases}$ indicates the Kronecker $\delta$.  We encourage the reader to take a moment and use this definition to compute $e_{12}(x_1x_2x_1^{-1}x_2^{-1})=1$.  The \textbf{triple linking number} is given by 
$$
\mu_{ijk}(L) = e_{ij}(w_k(F))+e_{jk}(w_i(F))+e_{ki}(w_j(F))
$$
When $L$ is a link with vanishing pairwise linking numbers, $\mu_{ijk}(L)$ is independent of the choice of $F$ and of the choice of basepoints.  

\begin{example}For the sake of clarity, we provide an example computing the triple linking number of the Borromean Rings $BR = BR_1\cup BR_2\cup BR_3$ using the C-complex $F$ of Figure~\ref{fig:examples}.


\begin{itemize}
\item Following $BR_1$ starting at the arrow we encounter in order a negative clasp with $F_3$, a positive clasp with $F_2$, a positive clasp with $F_3$ and a negative clasp with $F_2$.  Therefore, 
$$w(F_1) = x_3^{-1}x_2x_3x_2^{-1}.$$  
Similarly, 
$
w(F_2) = x_1^{-1}x_1
$
and 
$w(F_3) = x_1^{-1} x_1$.
\item Count with sign how many times you see $x_2$ before $x_3$ in $w(L_1)$ to get  $e_{23}(w_1(F)) = +1$.  Similarly, $e_{12}(w(F_3))=e_{31}(w(F_{2}))=0$.
\item The triple linking number is given by summing, 
$$\mu_{123}(BR) = e_{12}(w_3(F)) + e_{23}(w_1(F))+e_{31}(w_2(F)) = 1.$$
\end{itemize}
\end{example}

Our next goal is the statement and proof of Theorem~\ref{thm:eij in terms of area}, which computes $e_{ij}(w_k(F))$ in terms of some curve $\gamma_{ij}(w_k(F))$ in the plane.  We begin by explaining the construction of $\gamma_{ij}(w_k(F))$.  Let $w$ be any word in the letters $x_1^{\pm1}, \dots, x_n^{\pm1}$.  We give a procedure which associates to $w$ a curve in the plane.  Start at the point $(0,0)\in \R^2$.  Each time you encounter $x_i$ in $w$ travel right a length of $1$.  When $x_i^{-1}$ is encountered travel left.  When $x_j$ or $x_j^{-1}$ is encountered travel up or down respectively.  Call the resulting curve $\gamma_{ij}(w_k)$.   For instance, when $w = x_ix_jx_ix_jx_i^{-2}x_j^{-2}$, $\gamma_{ij}(w)$ appears in Figure~\ref{fig: polyExample}.  The assiduous reader will now compute $e_{ij}(w)=3$ using equation~\pref{eqn:eij} which suggestively agrees with the area of the region enclosed by $\gamma_{ij}(w)$.

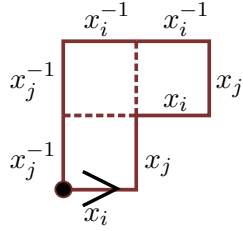
\begin{figure}[h]
\begin{picture}(65,70)
\put(0,0){\includegraphics[height=.15\textwidth]{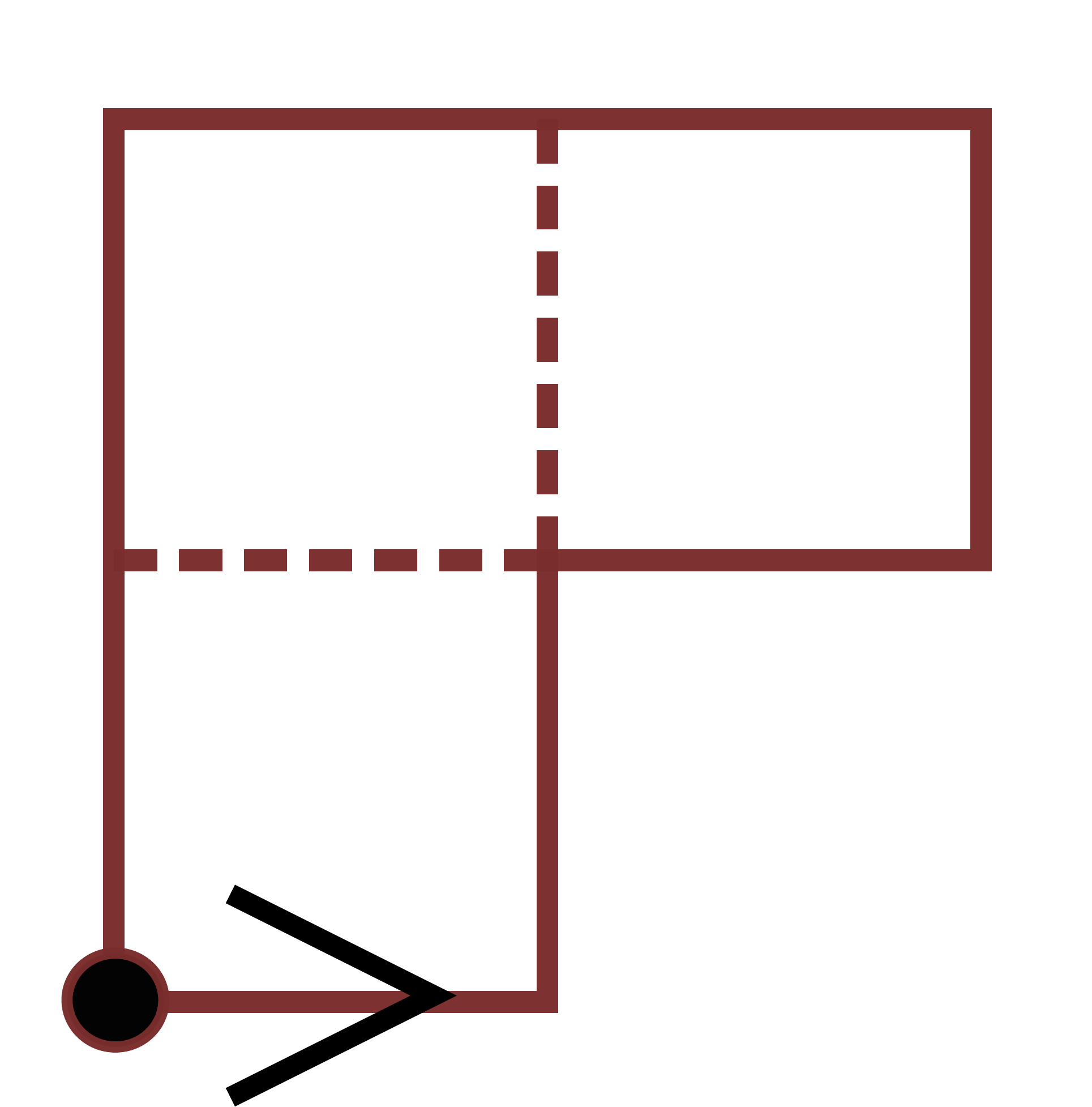}}
\put(15,-5){$x_i$}
\put(38,15){$x_j$}
\put(45,38){$x_i$}
\put(65,45){$x_j$}
\put(45,68){$x_i^{-1}$}
\put(15,68){$x_i^{-1}$}
\put(-13,45){$x_j^{-1}$}
\put(-13,15){$x_j^{-1}$}
\end{picture}
\caption{The curve $\gamma_{ij}(w)$ associated to the word $w = x_ix_jx_ix_jx_i^{-2}x_j^{-2}$ together with the region $\gamma_{ij}(w)$ encloses  encloses.  }
\label{fig: polyExample}
\end{figure}


We are now ready to prove Theorem~\ref{thm:eij in terms of area}.

\begin{reptheorem}{thm:eij in terms of area}
Let $w = \Prod_{v=1}^m x_{i_v}^{\epsilon_v}$ be a word in letters $x_1^{\pm1}, \dots, x_n^{\pm1}$.  For any $i\neq j\in \{1,\dots, n\}$, ${e_{ij}(w) = \Oint_{\gamma_{ij}(w)} x\,dy.}$
Additionally, if $\gamma_{ij}(w)$ is a simple closed curve with counterclockwise orientation, then $e_{ij}(w)$ is the area enclosed by $\gamma_{ij}(w)$.
\end{reptheorem}

\begin{proof}[Proof of Theorem~\ref{thm:eij in terms of area}]
Let $w = \Prod_{v=1}^m x_{i_v}^{\epsilon_v}$ be a word in the letters $x_1^\pm, \dots,  x_n^{\pm1}$.  Let $|w|=m$ be the length of $w$.  Then $\gamma_{ij}(w)$ consists of a concatenation of $|w|$ many curves, $\gamma_{ij}^1(w),\dots, \gamma_{ij}^m(w)$ where $\gamma_{ij}^v(w)$ is constant if $i_v\notin\{i,j\}$ and is a length 1 line segment traveling in a cardinal direction otherwise.  Therefore, the integral in question breaks up as 
$$
\Oint_{\gamma_{ij}(w)} x\,dy = \Sum_{v=1}^m\left( \Oint_{\gamma_{ij}^v (w)} x\,dy\right).
$$
If $i_v\neq j$ then $\gamma_{ij}^v (w)$ is either constant or parametrizes a horizontal line segment.  In either case $dy=0$ so that $\Oint_{\gamma_{ij}^v (w)} x\,dy=0$.  If $i_v=j$ then $\gamma_{ij}^v (w)$ is a vertical line segment parametrized by $\gamma_{ij}^v(t) = (x,t\cdot \epsilon_v+c)$ with $x$ and $c$ constants and $t$ running from $0$ to $1$.  In particular $dy=\epsilon_v\, dt$.  The fixed $x$-coordinate over which this vertical line sits is the signed count of $u<v$ with $i_u=i$:
$$
x=\Sum_{u=1}^v \delta(x_u,1)\epsilon_{u}.
$$
Thus, in the case that $i_v=j$, we have $\Oint_{\gamma_{ij}^v (w)} x\,dy =\Int_{0}^1 x\cdot \epsilon_v\, dt = x\cdot \epsilon_v = \Sum_{u=1}^v \delta(x_u,1)\epsilon_{u}\epsilon_v$.  Combining the cases $i_v = j$ and $i_v\neq j$, we see for for all $v$,  $$\Oint_{\gamma_{ij}^v (w)} x\,dy = \delta(i_v,j) \Sum_{u=1}^v \delta(x_u,1)\epsilon_{u}\epsilon_v.$$  Summing over all values of $v$, $\Oint_{\gamma_{ij}(w)} x\,dy = \Sum_{v=1}^n \delta(i_v,j) \Sum_{u=1}^v \delta(x_u,i)\epsilon_{u}\epsilon_v$.  An application of the distributive law reduces this to the definition of $e_{ij}(w)$ appearing in equation \pref{eqn:eij}.  This completes the proof of the first claim.

The second claim follows from a standard application of Green's theorem.  
\end{proof}

As an illustration of the efficacy of Theorem~\ref{thm:eij in terms of area} we use it to make some computations.

\begin{proposition}\label{prop:compute}
For any $n\in \N$, the generalized Boromean rings $BR^n$ of Figure~\ref{fig:examples} has triple linking number $n^2$.  
\end{proposition}
\begin{proof}
Using the C-complex of Figure~\ref{fig:examples}~(b) we get clasp words$$
w_1(F) = x_3^{-n}x_2^nx_3^nx_2^{-n}, w_2(F) = x_1^nx_1^{-n}, w_2(F) = (x_1x_1^{-1})^n
$$
The curve $\gamma_{23}(w_1(F))$ traces a counterclockwise $n\times n$ square.  The curve, $\gamma_{31}(w_2(F))$ lies in the vertical line $x=0$ so that $e_{31}(w_2(F)) = 0$.  Finally, $\gamma_{23}(w_1(F))$ lies in the horizontal line $y=0$ so that $e_{12}(w_3(F)) = 0$.  Therefore, $\mu_{123}(BR^n) = n^2$. 
\end{proof}

\section{The proof of Theorem~\ref{thm:main 3-component}}

We now turn our attention to a lower bound on the number of clasps in a C-complex in terms of the triple linking number.  Notice that the curve $\gamma_{ij}(w(L_k))$ of Section~\ref{sect: triple linking} has length equal to the number of clasps in $F_k\cap F_i$ plus the number of clasps in $F_k\cap F_j$.  By Theorem~\ref{thm:eij in terms of area}, $\Oint_{\gamma_{ij}(w)} x\,dy = e_{ij}(w(L_k))$.  Thus, we will begin the proof of Theorem~\ref{thm:main 3-component} by studying how $\Oint_{\gamma} x\,dy.$ provides a lower bound on the length of $\gamma$.  

For the lemma below, a \textbf{polyomino curve} is a closed curve in $\R^2$ given by a concatenation of straight lines of length 1 between points in $\Z^2$.  The length of a curve, $\gamma$, is denoted by $||\gamma||$.

\begin{lemma}\label{lem:length compared to eij}
Let $\gamma$ be a polyomino curve in $\R^2$.  Let $A = \Oint_{\gamma} x\,dy$.  Then $||\gamma||\ge 2\ceil{2\sqrt{|A|}}$.  
\end{lemma}

\begin{proof}
Let $\gamma$ be a polyomino curve in $\R^2$ and let $A = \Oint_{\gamma} x\,dy$.  If $\gamma$ is a simple closed curve then a standard application of Green's theorem shows that $|A| = \underset{A}{\displaystyle\iint} 1 \, dxdy$ is the area of the region $R$ enclosed by $\gamma$.  In \cite{HH76}, Harary-Harborth shows that the minimum perimeter amongst all polyominos with a fixed area $|A|$ is given by $\left(2\ceil{2\sqrt{|A|}}\right)$.  Thus, $||\gamma||$, which is the perimeter of $A$, is at least $2\ceil{2\sqrt{|A|}}$, as the lemma claims.

It remains to deal with the case that $\gamma$ is not simple.  Recall that by assumption, $\gamma$ consists of a concatenation of vertical and horizontal line segments of length 1.  Denote the rightward pointing horizontal line segments as $\gamma^r_1(t),\dots, \gamma^r_h(t)$, the leftward pointing as $\gamma^\ell_1(t),\dots, \gamma^\ell_h(t)$, the upward as $\gamma^u_1(t),\dots, \gamma^u_v(t)$ and the downward as $\gamma^d_1(t),\dots, \gamma^d_v(t)$.  As $\gamma$ is a closed curve, the number of rightward and leftward pointing segments must be equal to each other as must the number of upward and downward pointing segments. 

Up to a translation and a reparametrization preserving $||\gamma||$ and $\Oint_{\gamma} x\,dy$, we may assume that $\gamma$ is parametrized by some $(x(t),y(t))$ such that the minimum value of $x(t)$ is $x(0)=0$. It follows for all $t$ that $0\le x(t)\le h$, where $h$ is the number of rightward pointing length 1 line segments in $\gamma$.   Breaking the integral up as a sum,
\begin{equation}\label{eqn:ControlA}A = \Oint_{\gamma} x\,dy = \Sum_{i=1}^v\Oint_{\gamma^u_i}x\,dy+\Sum_{i=1}^v \Oint_{\gamma^d_i}x\,dy+\Sum_{i=1}^h \Oint_{\gamma^r_i}x\,dy+\Sum_{i=1}^h \Oint_{\gamma^\ell_i}x\,dy \end{equation}
Since $\gamma^\ell_i$ and $\gamma^r_i$ are horizontal line segments, they each have $dy=0$ so that $\Oint_{\gamma^r_i}x\,dy=\Oint_{\gamma^\ell_i}x\,dy=0.$  
Since $\gamma^u_i$ is an upward pointing length 1 line segment, we may parametrize $\gamma^u_i$ as $(x, t+c)$ where $x$ and $c$ are constant and $t$ runs from $0$ to $1$.  Therefore, $dy = dt$ and $0\le x\le h$. Thus, $\Oint_{\gamma_i^u}x\,dy = \Int_{0}^1x dt = x$ and in particular
$0\le \Oint_{\gamma_i^u}x\,dy \le h.$
Similarly, 
$-h\le \Oint_{\gamma_i^d}x\,dy\le 0$.
Therefore, $0\le \Sum_{i=1}^v \Oint_{\gamma_i^u}y\,dx \le h\cdot v$ and  $-h\cdot v \le \Sum_{i=1}^v \Oint_{\gamma_i^d}y\,dx \le 0$.  Applying these bounds  to the rightmost expression in \pref{eqn:ControlA}  we see that $-h\cdot v\le A\le h\cdot v$, so that $|A|\le h \cdot v$.  

Let $R$ be an $h\times v$ rectangle and let $r$ be the curve traversing its boundary counterclockwise.  As $r$ is made up of the same number of length 1 line segments as $\gamma$, $||\gamma||=||r||$.  Since $R$ is a polyomino of area $h\cdot v$, \cite{HH76} applies and  $||r||\ge 2\ceil{2\sqrt{h\cdot v}}$.  
 Summarizing,
$$
||\gamma|| = ||r||\ge 2\ceil{2\sqrt{h\cdot v}}\ge  2\ceil{2\sqrt{|A|}}.
$$ This completes the proof.

\end{proof}

If $w = \Prod_{v=1}^m x_{i_v}^{\epsilon_v}$ is a word in $x_1^{\pm1}, \dots, x_n^{\pm1}$ for which the signed count of $x_i$'s and $x_j$'s are both zero then $||\gamma_{ij}(w)||$ is the same as the length of the word $w$ after deleting all letters other than $x_i^{\pm1}$ and $x_j^{\pm1}$, while $e_{ij}(w) = \Oint_{\gamma_{ij}(w)} y\,dx$ by Theorem~\ref{thm:eij in terms of area}.  Thus, Lemma~\ref{lem:length compared to eij} has the following corollary.  

\begin{corollary}\label{cor:length compared to eij}
Let $w = \Prod_{n=1}^m x_{i_n}^{\epsilon_n}$ be a word in $x_1^{\pm1},\dots x_n^{\pm1}$.  Fix some $i\neq j\in \{1,\dots, n\}$ and assume the signed count of $x_i$'s and $x_j$'s are both zero.  If $e_{ij}(w)=A$ then $|w|\ge 2\ceil{2\sqrt{|A|}}$
\end{corollary}
%
%
%

We are now ready to prove Theorem~\ref{thm:main 3-component} giving a lower bound on $C(L)$ in terms of $\mu_{ijk}(L)$.

\begin{reptheorem}{thm:main 3-component}
Let $L=L_1\cup L_2\cup L_3$ be a 3-component link with vanishing pairwise linking numbers.  Then $C(L)\ge 2\ceil{2\sqrt{|\mu_{123}(L)|/3}}$.  
\end{reptheorem}
\begin{proof}

Let $L$ be a 3-component link with vanishing pairwise linking numbers and $F$ be a C-complex bounded by $L$.  Let $C(F)$ be the number of clasps between the components of $F$.    Let $w_1 = w_1(F)$, $w_2 = w_2(F)$ and $w_3=w_3(F)$ be the resulting clasp words.  Each clasps corresponds to a letter in two of these words, and so 
$$2C(F) = |w_1|+|w_2|+|w_3|.$$

Let $e_1=e_{23}(w_1)$, $e_2=e_{31}(w_2)$, and $e_3=e_{12}(w_3)$.  Then $\mu_{123}(L) = e_1+e_2+e_3$.  Assume without loss of generality that $|e_1|\le |e_2|\le |e_3|$.  Then it must be that $|e_3|\ge \frac{|\mu_{123}(L)|}{3}$.  Corollary \ref{cor:length compared to eij} concludes that $|w_3|\ge 2\ceil{2\sqrt{|e_3|}}\ge 2\ceil{2\sqrt{|\mu_{123}(L)|/3}}$.

Now, each letter of $w_3$ corresponds to either a clasp in $F_3\cap F_1$ or a clasp in $F_3\cap F_2$.  Each of these clasps produces a letter in $w_1$ or in $w_2$.  As a consequence $|w_3|\le |w_1|+|w_2|$.  Putting this together,
$$2C(F) = |w_1|+|w_2|+|w_3| \ge 2|w_3| \ge 4\ceil{2\sqrt{|\mu_{123}(L)|/3}}
$$
division by 2 completes the proof.  

\end{proof}

\bibliographystyle{plain}
\bibliography{biblio}  

\end{document}